\documentclass[11pt]{article}
\RequirePackage{amsthm,amsmath,amssymb,amsfonts}
\usepackage{graphicx}
\usepackage{color}
\usepackage{float}
\usepackage{subfigure}
\usepackage{psfrag,epsf}
\usepackage{enumerate}
\usepackage{natbib}
\usepackage{url} 

\makeatletter
\DeclareRobustCommand*\cal{\@fontswitch\relax\mathcal}
\makeatother
\newcommand{\bR}{\mathbb{R}} 
\newtheorem{theorem}{Theorem}[section]

\newtheorem{Lemma}[theorem]{Lemma}
\newtheorem{Assumption}{Assumption}[section]
\newtheorem{Remark}{Remark}
\begin{document}
\title{\bf An MM algorithm for estimation of a two component semiparametric density mixture
with a known component}
\author{Zhou Shen \\ {\small Department of Statistics, Purdue University}\\
  and 
    Michael Levine\thanks{
    Michael Levine gratefully acknowledges \textit{partial support from the NSF-DMS grant 1208994}} \\  {\small Department of Statistics, Purdue University}\\
    {\small 250 N. University St., West Lafayette, IN 47907}\\
  {\small  mlevins@purdue.edu} \hspace{.2cm}\\
      and Zuofeng Shang \\ {\small Department of Mathematical Sciences, SUNY Binghamton}
      }
    \maketitle


\begin{abstract}
We consider a semiparametric mixture of two univariate density functions where one of them is known while the weight and the other function are unknown. We do not assume any additional structure on the unknown density function. For this mixture model, we derive a new sufficient identifiability condition and pinpoint a specific class of distributions describing the unknown component for which this condition is mostly satisfied. We also suggest a novel approach to estimation of this model that is based on an idea of applying a maximum smoothed likelihood to what would otherwise have been an ill-posed problem. We introduce an iterative MM (Majorization-Minimization) algorithm that estimates all of the model parameters. We establish that the algorithm possesses a descent property with respect to a log-likelihood objective functional and prove that the algorithm, indeed, converges. Finally, we also illustrate the performance of our algorithm in a simulation study and using a real dataset. 
\end{abstract}

\vfill


\section{Introduction}\label{sec:intro}

We consider a general case of a two-component univariate mixture model where one component distribution is known while the mixing proportion and the other component distribution are unknown. Such a model can be defined at its most general as 
\begin{equation}\label{model2}
g(x)=(1-p)f_{0}(x)+pf(x)
\end{equation}
where $f_{0}$ is a known density component, while $p\in (0,1)$ and $f(x)$ are the unknown weight and the unknown density component, respectively. The semiparametric mixtures of density functions have been considered by now in a number of publications. The earliest seminal publications in this area are \cite{Hall_Zhou} and \cite{Hall_Pakyari_Elmore}.
From the practical viewpoint, the model \eqref{model2} is related to the multiple testing problem where $p$-values are uniformly distributed on $[0,1]$ under the null hypothesis but their distribution under the alternative is unknown. In the setting of model \eqref{model2}, this means that the known distribution is uniform while the goal is to estimate the proportion of the false null hypothesis $p$ and the distribution of the $p-$ values under the alternative. More detailed descriptions in statistical literature can be found in e.g. \cite{efron2012large} and \cite{robin2007semi}. 
Historically, whenever a two-component mixture model with a known component was considered, some assumptions were imposed on the unknown density function $f(x)$. Most commonly, it was assumed that an unknown distribution belongs to a particular parametric family. In such a case, \cite{cohen1967estimation} and \cite{lindsay1983geometry} used the maximum likelihood-based method  to fit it; \cite{day1969estimating} used the minimum $\chi^{2}$ method, while \cite{lindsay1993multivariate} used the method of moments. \cite{jin2008proportion} and \cite{cai2010optimal} used empirical characteristic functions to estimate the unknown cumulative density function under a semiparametric normal mixture model. A somewhat different direction was taken by \cite{bordes_delmas} who considered a special case of the model \eqref{model2} where the unknown component belonged to a location family. In other words, their model is defined as 
\begin{equation}\label{model1}
g(x)=(1-p)f_{0}(x)+pf(x-\mu)
\end{equation} where $f_{0}$ is known while $p\in (0,1)$, the non-null location parameter $\mu$ and the pdf $f$ that is symmetric around $\mu$ are the unknown parameters. The model of \cite{bordes_delmas} was motivated by the problem of detection of differentially expressed genes under two or more conditions in microarray data. Typically, a test statistic is built for each gene. Under the null hypothesis (which corresponds to a lack of difference in expression), such a test statistic has a {\it known} distribution (commonly, Student's or Fisher). Then, the response of thousands of genes is observed; such a response can be thought of as coming from a mixture of two distributions: the known distribution $f_0$ (under the null hypothesis) and the unknown distribution $f$ under the alternative hypothesis. Once the parameters $p$, $\mu$, and $f$ are estimated, we can estimate the probability that a gene belongs to a null component distribution conditionally on observations. 

\cite{bordes_delmas} establishes some sufficient identifiability conditions for their proposed model; they also suggest two estimation methods for it, both of which rely heavily on the fact that the density function of the unknown component is symmetric. A sequel paper, \cite{bordes2010semiparametric}, also establishes a joint central limit theorem for estimators that are based on one of these methods.  There is no particular practical reason to make the unknown component symmetric and \cite{bordes_delmas} themselves note that ``In our opinion, a challenging problem would be to consider model \eqref{model1} without the symmetry assumption on the unknown component". This is the goal we set for ourselves in this manuscript. Our approach is  based on, first, defining the (joint) estimator of $f$ and $p$ as a minimizer of the log-likelihood type objective functional of $p$ and $f$. Such a definition is implicit in nature; however, we construct an MM (Majorization-Minimization) iterative algorithm that possesses  a descent property with respect to that objective functional. Moreover, we also show that the resulting algorithm actually converges. Our simulation studies also show that the algorithm is rather well-behaved in practice.  

Just as we were finishing our work, a related publication \citet{patra2015estimation} came to our attention. \citet{patra2015estimation} also consider a two-component mixture model with one unknown component. Their identifiability approach is somewhat more general as our discussion mostly concerns sufficient conditions for specific functional classes containing the function $f$. They propose some general identifiability criteria for this model and obtain a separate estimator of the weight $p$; moreover, they also construct a distribution free finite sample lower confidence bound for the weight $p$. \citet{patra2015estimation} start with estimating parameter $p$ first; then, a completely automated and tuning parameter free estimate of $f$ can be constructed when $f$ is decreasing. When $f$ is not decreasing, one can start with e.g. estimating $\hat g$ based on observations $X_{1},\ldots,X_{n}$; then, one can construct e.g. a kernel estimate of unknown $f$ that is proportional to $\max(\hat g-(1-\hat p)f_{0},0)$.  In contrast to their approach, our approach estimates both $p$ and $f$ jointly and the algorithm works the same way regardless of the shape of $f$.  

The rest of our manuscript is structured as follows. Section \eqref{ident} discusses identifiability of the model \eqref{model2}. 
Section \eqref{est} introduces our approach to estimation of the model \eqref{model2}. Section \eqref{emp_sec} suggests an empirical version of the algorithm first introduced in Section \eqref{est} that can be implemented in practice. Section \eqref{sim} provides simulated examples of the performance of our algorithm. Section \eqref{realdata} gives a real data example. Finally, Section \eqref{Discussion} rounds out the manuscript with a discussion of our result and delineation of some related future research.

\section{Identifiability}\label{ident}

In general, the model \eqref{model2} is not identifiable. In what follows, we investigate some special cases. For an unknown density function $f$, let us denote  its mean $\mu_{f}$ and its variance $\sigma^{2}_{f}$. To state a sufficient identifiability result, we consider a general equation 
\begin{equation}\label{unid}
(1-p)f_{0}(x)+pf(x)=(1-p_{1})f_{0}(x)+p_{1}f_{1}(x).
\end{equation}
We also denote variance of the distribution $f(x)$ as a function of its mean $\mu_{f}$ as $V(\mu_{f})$. 
\begin{Lemma}\label{identifiability}
Consider the model \eqref{model2} with the unknown density function $f$. Without loss of generality, assume that the first moment of $f_{0}$ is zero while its second  moment is finite. We assume that the function $f$ belongs to a set of density functions whose first two moments are finite, whose means are not equal to zero and that are all of the same sign; that is, $f\in {\cal F}=\{f:\int x^{2}f(x)\,dx <+\infty; \mu_{f}>0 \mbox{ or } \mu_{f}<0\}$. Moreover, we assume that for any $f\in {\cal F}$ the function $G(\mu_f)=\frac{V(\mu_f)}{\mu_f}$ is strictly increasing. Then, the equation \eqref{unid} has the unique solution $p_{1}=p$ and $f_{1}=f$. 
\end{Lemma}

\begin{proof}
First, let us assume that the mean $\mu_{f}> 0$. 
Then, the assumption of our lemma implies that the function $V: (0,\infty)\mapsto (0,\infty)$ is strictly increasing. Let us use the notation $\theta_{0}$ for the second moment of $f_{0}$. If we assume that there are distinct $p_{1}\ne p$ and $f_{1}\ne f$ such that $(1-p)f_{0}(x)+pf(x)=(1-p_{1})f_{0}(x)+p_{1}f_{1}(x)$, the following two moment equations are easily obtained  
\begin{equation}\label{1st}
\zeta=p_1\mu_{f_1}=p\mu_f
\end{equation}
and 
\begin{equation}\label{2nd}
(p_1-p)\theta_0=\zeta(\mu_{f_1}-\mu_f)+p_1V(\mu_{f_1})-pV(\mu_f),
\end{equation} 
where $\zeta>0$. Our task is now to show that if \eqref{1st} and \eqref{2nd} are true, then $p=p_1$ and $f=f_1$.
To see this, let us assume $p_1>p$ (the case $p_1<p$ can be treated in exactly the same way). Then from the first equation we have immediately that $\mu_{f_1}<\mu_f$; moreover, since the function $G(\mu_f)$ is a strictly increasing one, then so is the function $\mu_f+G(\mu_f)$. With this in mind, we have   
\[
\mu_{f_1}+\frac{V(\mu_{f_1})}{\mu_{f_1}}<\mu_{f}+\frac{V(\mu_{f})}{\mu_{f}}.
\]
On the other hand, $(p_1-p)\theta_0\ge0$
which implies 
\[
0\le\zeta(\mu_{f_1}-\mu_f)+p_1V(\mu_{f_1})-pV(\mu_f)=\zeta(\mu_{f_1}-\mu_f)+\zeta
\left(\frac{V(\mu_{f_1})}{\mu_{f_1}}-\frac{V(\mu_{f})}{\mu_{f}}\right).
\]
 Therefore, this implies that  
\[
\mu_{f_1}+\frac{V(\mu_{f_1})}{\mu_{f_1}}\ge\mu_{f}+\frac{V(\mu_{f})}{\mu_{f}}.
\]
 and we end up with a contradiction. Therefore, we must have $p=p_1$. This, in turn, implies immediately that 
$f=f_1$.

The case where $\mu_{f}<0$ proceeds similarly. Let us now consider the case where the variance function $V:(-\infty,0)\rightarrow (0,\infty)$ and is strictly monotonically increasing. As a first step, again take $p_{1}>p$. Clearly, the first moment equation is yet again
\eqref{1st} where now $\zeta<0$. If $p_{1}>p$, we now have $\mu_{f_{1}}>\mu_{f}$ and, due to the strict monotonicity of $G(\mu)$, we have  $\mu_{f_{1}}+\frac{V(\mu_{f_{1}})}{\mu_{f_{1}}}>\mu_{f}+\frac{V(\mu_{f})}{\mu_{f}}$. On the other hand, since $(p_{1}-p)\theta_{0}\ge 0$, we have 
\begin{align*}
&0\le \zeta(\mu_{f_{1}}-\mu_{f})+p_{1}V(\mu_{f_{1}})-pV(\mu_{f})\\
&=\zeta\left(\left\{\mu_{f_{1}}+\frac{V(\mu_{f_{1}})}{\mu_{f_{1}}}\right\}-\left\{\mu_{f}+\frac{V(\mu_{f})}{\mu_{f}}\right\}\right).
\end{align*}
Because $\zeta<0$, the above implies that $ \left\{\mu_{f_{1}}+\frac{V(\mu_{f_{1}})}{\mu_{f_{1}}}\right\}-\left\{\mu_{f}+\frac{V(\mu_{f})}{\mu_{f}}\right\}<0$ which contradicts the assumption that the function $G(\mu)$ is strictly increasing. 
\end{proof}

\begin{Remark}
To understand better what is going on here,  it is helpful if we can suggest a more specific density class which satisfies the sufficient condition in Lemma \eqref{identifiability}. The form of Lemma \eqref{identifiability} suggests one such possiibility - a family of natural exponential families with power variance functions (NEF-PVF).  For convenience, we give the definition due to \cite{Bar-Lev_Stramer1987}: ``A natural exponential family (NEF for short) is said to have a power variance function if its variance function is of the form $V(\mu)=\alpha\mu^{\gamma}$, $\mu\in \Omega$, for some constants $\alpha\ne0$ and $\gamma$, called the scale and power parameters, respectively". This family of distributions is discussed in detail in \cite{bar1986reproducibility} and \cite{Bar-Lev_Stramer1987}. In particular, they establish that the parameter space $\Omega$ can only be ${\bR}$, $\bR^{+}$ and $\bR^{-}$; moreover, we can only have $\gamma=0$ iff $\Omega=\bR$. The most interesting for us property is that (see Theorem 2.1 from \cite{Bar-Lev_Stramer1987} for details) is that for any NEF-PVF, it is necessary that $\gamma\notin (-\infty,0)\cup(0,1)$; in other words, possible values of $\gamma$ are $0$, corresponding to the normal distribution, $1$, corresponding to Poisson, and any positive real numbers that are greater than $1$. In particular, the case $\gamma=2$ corresponds to gamma distribution. Out of these choices, the only one that does not result in a monotonically increasing function $G(\mu)$ is $\gamma=0$ that corresponds to the normal distribution; thus, we have to exclude it from consideration. With this exception gone, the NEF-PVF framework includes only density families with either strictly positive or strictly negative means; due to this, it seems a rather good fit for the description of the family of density functions $f$ in the Lemma \eqref{identifiability}. 

Note that the exclusion of the normal distribution is also rather sensible from the practical viewpoint because it belongs to a location family; therefore, it can be treated in the framework of \cite{bordes_delmas}. More specifically, Proposition $1$ of \cite{bordes_delmas} suggests that, when $f(x)$ is normal,  the equation \eqref{unid} has at most two solutions if $f_{0}$ is an even pdf and at most three solutions if $f_{0}$ is not an even pdf. 
\end{Remark}

\begin{Remark}
It is also of interest to compare our Lemma \eqref{identifiability} with the Lemma 4 of \citet{patra2015estimation} that also establishes an identifiability result for the model \eqref{model2}. The notions of identifiability that are considered in the two results differ: whereas we discuss the identifiability based on the first two moments, Lemma 4 of \citet{patra2015estimation} looks at a somewhat different definition of identifiability. At the same time, the interpretation given in the previous Remark, suggests an interesting connection. For example, the case where the unknown density function $f$ is gamma corresponds to the power parameter of the NEF-PVF family being equal to $2$. According to our identifiability result Lemma \eqref{identifiability}, the mixture model \eqref{model2} is, then, identifiable with respect to the first two moments. On the other hand, let us assume that the known density function $f_{0}$ is the standard normal. Since its support fully contains the support of any density from the gamma family, identifiability in the sense of \citet{patra2015estimation} now follows from their Lemma 4.  
\end{Remark}

\begin{Remark}
We only assumed that the first moment of $f_{0}$ is equal to zero for simplicity. It is not hard to reformulate the  Lemma \eqref{identifiability} if this is not the case. The proof is analogous. 
\begin{Lemma}\label{identifiability_1}
Consider the model \eqref{model2} with the unknown density function $f$. We assume that the known density $f_{0}$ has finite first two moments and denote its first moment $\mu_{f_{0}}$. We also assume that the function $f$ belongs to a set of density functions whose first two moments are finite, and whose means are all either greater than $\mu_{f_{0}}$ or less than $\mu_{f_{0}}$: 
\[
f\in {\cal F}=\{f:\int x^{2}f(x)\,dx<+\infty; \mu_{f}>\mu_{f_{0}} \mbox{ or } \mu_{f}<\mu_{f_{0}}\}.
\]
Let us assume that $G(\mu_{f})=\frac{V(\mu_{f})}{\mu_{f}-\mu_{f_{0}}}$ is a strictly increasing function in $\mu_{f}$ for a fixed, known $f_{0}$.  Then, the equation \eqref{unid} has the unique solution $p_{1}=p$ and $f_{1}=f$.  
\end{Lemma}
\end{Remark}



\section{Estimation}\label{est}

\subsection{Possible interpretations of our approach} Let $h$ be a positive bandwidth and $K$ a symmetric positive-valued kernel function that is also a true density; as a technical assumption, we will also assume that $K$ is continuously differentiable. The rescaled version of this kernel function is denoted $K_h(x)=K(x/h)/h$ for any $x\in\bR$. We will also need a linear smoothing operator $\mathcal{S}f(x)=\int K_h(x-u)f(u)du$ and a nonlinear smoothing operator $\mathcal{N}f(x)=\exp(\mathcal{S}\log{f(x)})$
for any generic density function $f$. For simplicity, let us assume that our densities are defined on a closed interval, e.g. $[0,1]$. This assumption is here for technical convenience only when proving algorithmic convergence related results. Simulations in the Section \eqref{sim} show that the algorithm works well also when the support of the density $f$ is e.g. half the real line. In the future, we will omit these integration limits whenever doing so doesn't cause confusion. A simple application of Jensen's inequality and Fubini's theorem suggests that $\int{\cal N}f(x)\,dx\le \int Sf(x)\,dx=\int f(x)\,dx=1$. 

Our estimation approach is based on selecting $p$ and $f$ that minimize the following log-likelihood type objective functional:
\begin{equation}\label{obj.fc1}
\ell(p,f)=\int g(x)\log\frac{g(x)}{(1-p)f_0(x)+p\mathcal{N}f(x)}dx.
\end{equation}
The reason the functional \eqref{obj.fc1} is of interest as an objective functional is as follows. First, recall that $KL(a(x),b(x))=\int \left[a(x)\log \frac{a(x)}{b(x)}+b(x)-a(x)\right]\,dx$ is a Kullback-Leibler distance between the two arbitrary positive integrable functions (not necessarily densities) $a(x)$ and $b(x)$; as usual, $KL(a,b)\ge 0$. Note that 
the functional \eqref{obj.fc1} can be represented as a penalized Kullback-Leibler distance between the target density $g(x)$ and the smoothed version of the mixture $(1-p)f_{0}(x)+p{\cal N}f(x)$; indeed, we can represent $\ell(p,f)$ as 
 \begin{equation}\label{main.problem}
 \ell (p,f)=KL(g(x),(1-p)f_{0}(x)+p{\cal N}f(x))+p\left\{1-\int {\cal N}f(x)\,dx\right\}.
 \end{equation}
The quantity $1-\int {\cal N}f(x)\,dx=\int [f(x)-{\cal N}f(x)]\,dx$ is effectively the penalty on the smoothness of the unknown density. Thus, the functional \eqref{obj.fc1} can be interpreted as a penalized smoothed likelihood functional. 

Of course, it is a matter of serious interest to find out if the optimization problem \eqref{main.problem} has a solution at all.  This problem can be thought of as a kind of generalized Tikhonov regularization problem; these problems have recently become an object of substantial interest in the area of ill-posed inverse problems. A very nice framework for these problems is described in the monograph \cite{flemming2011generalized} and we will follow it here. First of all, we define the domain of the operator ${\cal N}$ to be a set of square integrable densities, t.i. all densities on $[0,1]$ that belong in $L_{2}[0,1]$. We also define $L_{2}^{+}([0,1])$ to be the set of all non-negative functions that belong to $L_{2}([0,1])$. Define a nonlinear smoothing operator $A:{\cal D}(A)\subseteq L_{2}(D)\rightarrow L_{2}(D)$ as
\begin{align*}
Af(x):=(1-p)f_{0}(x)+p{\cal N}f(x)
\end{align*}
where ${\cal D}(A)=\{f(x):f \in L_{2}^{+}([0,1]),\ f(x)\geq \eta >0,\ \int_{0}^{1} f(x)\,dx=1,\ \exists F\in \mathbb{R}^+\ s.t.\ ||f||_2\leq F\}$. In optimization theory, $A$ is commonly called a {\it forward operator}. Note that, as long as $||K||_2^{2}:=\int K^{2}(u)\,du<\infty$, it is easy to show that ${\cal N}f(x)\in L_{2}([0,1])$ if $f(x)\in L_{2}([0,1])$ and, therefore, this framework is justified. 

Next, for any two functions $a(x)$ and $b(x)$, we define a {\it fitting functional}  $Q:L_2([0,1])\times L_2([0,1])\rightarrow [0,\infty)$ as $Q(a(x),b(x));= KL(a(x),b(x))$. Finally, we also define a non-negative {\it stabilizing functional} $\Omega :{\cal D}(\Omega)\subseteq L_2([0,1])\rightarrow [0,1]$ as $\Omega(f) := \left\{1-\int_{0}^{1} {\cal N}f(x)\,dx\right\}$ where ${\cal D}(\Omega)={\cal D}(A)$. Now, we are ready to define the minimization problem
\begin{equation}\label{equ:min_problem}
T_{p,g}(f)=Q_p(g,Af)+p\Omega(f)\rightarrow \min
\end{equation}
where $p$ plays the role of {\it regularization parameter}. We use the subscript $p$ for $Q$ to stress the fact that the fitting functional is dependent on the regularization parameter; this doesn't seem to be common in optimization literature but we can still obtain the existence result that we need. The following set of assumptions is needed to establish existence of the solution of this problem; although a version of these assumptions is given in \cite{flemming2011generalized} , we give them here in full for ease of reading. 
\begin{Assumption}\label{assumption}
Assumptions on $A: {\cal D}(A)\subset L_{2}([0,1])\rightarrow L_{2}([0,1])$
\begin{enumerate}
\item\label{A1} $A$ is sequentially continuous with respect to the weak topology of the space $L_{2}([0,1])$, i.e. if $f_{k}\rightharpoonup f$ for $f,f_{k}\in {\cal D}(A)$, then we have $A(f_{k})\rightharpoonup A(f)$
\item\label{A2} ${\cal D}(A)$ is sequentially closed with respect to the weak topology on $L_{2}([0,1])$. This means that $f_{k}\rightharpoonup f$ for $\{f_{k}\}\in {\cal D}(A)$  implies that $f\in {\cal D}(A)$. 
\end{enumerate}
Assumptions on $Q:L_{2}([0,1])\times L_{2}([0,1])\rightarrow [0,\infty)$:
\begin{enumerate}\setcounter{enumi}{2}
\item \label{Q1}$Q_p(g,v)$ is sequentially lower semi-continuous with respect to the weak topology on $L_{2}([0,1])\times L_{2}([0,1])$, that is if $p_{k}\rightarrow p$, $g_{k}\rightharpoonup g$ and $v_{k}\rightharpoonup v$, then $Q_p(g,v)\le \liminf_{k\rightarrow \infty} Q_{p_k}(g_k,v_k)$.
\item\label{Q2} If $Q_p(g,v_{k})\rightarrow 0$ then there exists some $v\in L_{2}([0,1])$ such that $v_{k}\rightharpoonup v$. 
\item\label{Q3} If $v_{k}\rightharpoonup v$ and $Q_p(g,v)<\infty$, then $Q_p(g,v_{K})\rightarrow Q_p(g,v)$.
\end{enumerate}
Assumptions on  $\Omega:{\cal D}(A)\rightarrow [0,1]$:
\begin{enumerate}\setcounter{enumi}{5}
\item\label{O1} $\Omega$ is sequentially lower semicontinuous with respect to the weak topology in $L_{2}([0,1])$, that is, if $f_{k}\rightharpoonup f$ for $f,f_{k}\in L_{2}([0,1])$, we have $\Omega(f)\le \lim\inf_{k\rightarrow \infty}\Omega(f_{k})$
\item\label{O2} The sets $M_{\Omega}(c):=\{f\in {\cal D}(A):\Omega(f)\le c\}$ are sequentially pre-compact with respect to the weak topology on $L_{2}([0,1])$ for all $c\in \bR$, that is each sequence in $M_{\Omega}(c)$ has a subsequence that is convergent in the weak topology on $L_{2}([0,1])$. 
\end{enumerate}
\end{Assumption}
\begin{Lemma}
Assume that the kernel function $K$ is bounded from above: $K(x)\le {\cal K}$. Then, the optimization problem \eqref{equ:min_problem} satisfies all of the assumptions listed in \eqref{assumption}. 
\end{Lemma}
\begin{proof}
We start with the Assumption \ref{assumption}({\romannumeral 0\ref{A1}}). Note that the space dual to $L_{2}([0,1])$ is again $L_{2}([0,1])$; therefore, the weak convergence $f_{k}\rightharpoonup f$ in $L_2([0,1])$ means that, for any $q\in L_{2}([0,1])$, we have $\int_{0}^{1}f_{k}(x)q(x)\,dx\rightarrow \int_{0}^{1}f(x)q(x)\,dx$ as $k\rightarrow \infty$. To show that the Assumption \ref{assumption}({\romannumeral 0\ref{A1}}) is, indeed, true, we first note that $\{f_k\}$ and $f$ are bounded away from 0 which tells $\int_{0}^{1} |\log f_k(x) -\log f(x)|q(x)dx\leq \int_{0}^{1} |f_k(x)-f(x)|\frac{1}{\eta}q(x)dx\rightarrow 0$ as $k\rightarrow \infty$ for some positive $\eta$ that does not depend on $k$. Therefore, $f_k\rightharpoonup f$ implies $\log f_k\rightharpoonup \log f$. Second, 
\begin{align*}
\int S\log f_k(x)q(x)dx &= \int q(x)\int_{0}^{1} K_h(x-u)\log f_k(u)du\,dx \\
&= \int_{0}^{1} \log f_k(u)\int K_h(x-u)q(x)dx\,du
\end{align*}
Note that, since $\log f_k\rightharpoonup \log f$, and the function $\tilde q(u)=\int K_h(x-u)q(x)dx$ belongs to $L_{2}([0,1])$, we can claim that 
$\int S\log f_k(x)q(x)dx \longrightarrow \int_{0}^{1} \log f(u)\int K_h(x-u)q(x)dx\,du=\int S\log f(x)q(x)dx$ as $k\rightarrow \infty$. In other words, we just established that $S\log f_{k}\rightharpoonup S\log f$ as $k\rightarrow \infty$. Moving ahead, we find out, using the Cauchy - Schwarz inequality that $\int_{0}^{1} K_h(x-u)\log f_k(u)du < \int_{0}^{1} K_h(x-u) f_k(u)du \leq {\cal K}\int_{0}^{1} f_k(u)du = {\cal K}$. The same is true for $f(x)$ and so  $\int |exp\{ S\log f_k(x)\} - exp\{ S\log f(x)\}|g(x)dx \leq \int |S\log f_k(x)-S\log f(x)|\le E\cdot g(x)dx \rightarrow 0$ where $E$ is a positive constant that does not depend on $k$. Therefore, $f_k\rightharpoonup f$ finally implies ${\cal N} f_k \rightharpoonup {\cal N} f$ and thus $Af_k \rightharpoonup Af$.

Next, we need to prove that the assumption \ref{assumption}({\romannumeral 0\ref{A2}}) is also valid. To show that ${\cal D}(A)$ is sequentially closed, we select a particular function $q\equiv 1$ on $[0,1]$. Such a function clearly belongs to $L_{2}([0,1])$  and so we have  $\int_{0}^{1} f_{k}(x)q(x)\,dx\equiv \int_{0}^{1} f_{k}(x)\,dx \rightarrow \int_{0}^{1}f(x)\,dx$ as $k\rightarrow \infty$. Since we know that, for any $k$, we have $\int f_{k}(x)\,dx=1$, it follows that $\int_{0}^{1} f(x)dx=1$ as well. It is not hard to check that other characteristics of $D(A)$ are preserved under weak convergence as well. 

The fitting functional $Q$ is a Kullback-Leibler functional; the fact that it satisfies assumptions \ref{assumption}({\romannumeral 0\ref{Q1}})({\romannumeral 0\ref{Q2}})({\romannumeral 0\ref{Q3}}) has been demonstrated several times in optimization literature concerned with variational regularization with non-metric fitting functionals. The details can be found in e.g. \cite{flemming2010theory}. 

The sequential lower semi-continuity of the stabilizing functional $\Omega$ in \ref{assumption}({\romannumeral 0\ref{O1}}) is guaranteed by the weak version of Fatou's Lemma. Indeed, let us define $\phi_k(x)=Sf_k(x)-{\cal N}f_k(x)$. Then, due to Jensen's inequality, $\{\phi_k\}$ is a sequence of non-negative measurable functions. We already know that $f_{k}\rightharpoonup f$ guarantees ${\cal N}f_{k}\rightharpoonup {\cal N}f$; therefore, we have $\phi_{k}\rightharpoonup \phi$ where $\phi (x)=Sf(x)-{\cal N}f(x)$. By the weak version of Fatou's lemma, we then have $\int \phi (x)\,dx\le \liminf_{k\rightarrow \infty}\int \phi_k(x)\,dx$, or equivalently, $\Omega (f)\le \liminf_{k\rightarrow \infty}\Omega (f_k)$. Therefore, $\Omega :{\cal D}(A)\rightarrow [0,1]$ is lower semi-continuous with respect to the weak topology on $L_{2}([0,1])$.
Finally, the assumption \ref{assumption}({\romannumeral 0\ref{O2}}) is always true simply because $D(A)$ is a closed subset of a closed ball in $L_{2}([0,1])$; sequential Banach-Alaoglu theorem lets us conclude then that $M_{\Omega}(c)$ is sequentially compact with respect to the weak topology on $L_{2}([0,1])$.
\end{proof}

Finally, we can state the existence result. Note that in optimization literature sometimes one can see results of this nature under the heading of {\it well-posedness}, not existence; see, e.g. \cite{hofmann2007convergence}. 
\begin{theorem}\label{thm:Existence}
{\bf (Existence)}
For any mixture density $g(x)\in L_2([0,1])$ and any $0<p<1$, the minimization problem \eqref{equ:min_problem} has a solution. The minimizer $f^*\in {\cal D}(A)$ satisfies $T_{p,g}(f^*)< \infty$ if and only if there exists a density function $\bar{f}\in {\cal D}(A)$ with $Q_p(g,A\bar{f})< \infty$.
\end{theorem}

\begin{proof}
Set $c:=\inf_{f\in {\cal D}(A)}T_{p,g}(f)<\infty$. Note that $c<\infty$ due to existence of $\bar{f}$ and thus the trivial case of $c=\infty$ is excluded. Next, take a sequence $(f_k)_{k\in \mathbb{N}}\in {\cal D}(A)$ with $T_{p,g}(f_k) \rightarrow c$. Then
\begin{equation}
\Omega(f_k) \le \frac{1}{p}T_{p,g}(f_k)\le \frac{1}{p}(c+1)
\end{equation}
for sufficiently large $k$ and by the compactness of the sublevel sets of $\Omega$ there is a subsequence $(f_{k_l})_{l\in \mathbb{N}}$ converging to some $\tilde{f}\in {\cal D}(A)$. The continuity of $A$ implies $Af_{k_l}\rightharpoonup A\tilde{f}$ and the lower semi-continuity of $Q_p$ and $\Omega$ gives
\begin{equation}
T_{p,g}(\tilde{f}) \le \liminf_{l \rightarrow \infty} T_{p,g}(f_{k_l})=c
\end{equation}
that is, $\tilde{f}$ is a minimizer of $T_{p,g}$.
\end{proof}

\subsection{Algorithm}
Now we go back to introducing the algorithm that would search for unknown $p$ and $f(x)$. The first result that we need is the following technical Lemma.  
\begin{Lemma}\label{lemma:divergence}
For any pdf $\widetilde{f}$ and any real number $\widetilde{p}\in(0,1)$,
\begin{align}\label{eqn:divergence}
&\ell(\widetilde{p},\widetilde{f})-\ell(p,f)\\
&\le-\int g(x)\left[(1-w(x))\log\left(\frac{1-\widetilde{p}}{1-p}\right)
+w(x)\log\left(\frac{\widetilde{p}\mathcal{N}\widetilde{f}(x)}{p\mathcal{N}f(x)}\right)\right]dx\nonumber
\end{align}
where $w(x)=\frac{p\mathcal{N}f(x)}{(1-p)f_0(x)+p\mathcal{N}f(x)}$.
\end{Lemma}

\begin{proof}[Proof of Lemma \ref{lemma:divergence}]
The result follows by the following straightforward calculations:
\begin{eqnarray*}
&\ell(\widetilde{p},\widetilde{f})-\ell(p,f)=
-\int g(x)\log\left(\frac{(1-\widetilde{p})f_0(x)+\widetilde{p}\mathcal{N}\widetilde{f}(x)}
{(1-p)f_0(x)+p\mathcal{N}f(x)}\right)dx\\
&=-\int g(x)\log\left((1-w(x))\frac{1-\widetilde{p}}{1-p}+
w(x)\frac{\widetilde{p}\mathcal{N}\widetilde{f}(x)}{p\mathcal{N}f(x)}\right)dx\\
&\le -\int g(x)\left[(1-w(x))\log\left(\frac{1-\widetilde{p}}{1-p}\right)+w(x)
\log\left(\frac{\widetilde{p}\mathcal{N}\widetilde{f}(x)}{p\mathcal{N}f(x)}\right)\right],
\end{eqnarray*}
where the last inequality follows by convexity of the negative logarithm function.
\end{proof}

Suppose at iteration $t$, we get the updated pdf $f^t$ and the updated mixing proportion $p^t$.
Let $w^{t}(x)=\frac{p^{t}\mathcal{N}f^{t}(x)}{(1-p^t)f_0(x)+p^t\mathcal{N}f^t(x)}$, and define 
\[
p^{t+1}=\int g(x)w^t(x)dx,
\] 
\[
f^{t+1}(x)=\alpha^{t+1}\int K_h(x-u)g(u)w^t(u)du,
\]
where $\alpha^{t+1}$ is a normalizing constant needed to ensure that $f^{t+1}$ integrates to one.
Then the following result holds.

\begin{theorem}\label{descent:property}
For any $t\ge0$, $\ell(p^{t+1},f^{t+1})\le\ell(p^t,f^t)$.
\end{theorem}

\begin{proof}[Proof of Theorem \ref{descent:property}]

By Lemma \ref{lemma:divergence}, for an arbitrary density function $\widetilde f$ and an arbitrary number $0<\widetilde p<1$
\begin{align}\label{descent:property:eqn1}
&\ell(\widetilde{p},\widetilde{f})-\ell(p^t,f^t)\\
&\le -\int g(x)\left[(1-w^t(x))\log\left(\frac{1-\widetilde{p}}{1-p^t}\right)
+w^t(x)\log\left(\frac{\widetilde{p}\mathcal{N}\widetilde{f}(x)}{p^t\mathcal{N}f^t(x)}\right)\right]dx\nonumber.
\end{align}

Let $(\widehat{p},\widehat{f})$ be the minimizer of the right hand side of (\ref{descent:property:eqn1}) with respect to $\widetilde p$ and $\widetilde f$. Note that the right-hand side becomes zero when $\widetilde{p}=p^{t}$ and $\widetilde{f}=f^{t}$; therefore, the minimum value of the functional on the right hand side must be less then or equal to $0$. Therefore, it is clear that $\ell(\widehat{p},\widehat{f})\le\ell(p^t,f^t)$. To verify that the statement of the theorem \eqref{descent:property} is true, it remains only to show that $(\widehat{p},\widehat{f})=(p^{t+1},f^{t+1})$.

Note that the right hand side of (\ref{descent:property:eqn1}) can be rewritten as
\begin{align*}
&-\int g(x)[(1-w^t(x))\log(1-\widetilde{p})+w^t(x)\log\widetilde{p}]dx\\
&-\int g(x)w^t(x)\log\mathcal{N}\widetilde{f}(x)dx+T,
\end{align*}
where the term $T$ only depends on $(p^t,f^t)$. The first integral in the above only depends on 
$\widetilde{p}$ but not on $\widetilde{f}$. It is easy to see that the minimizer of this first integral with respect to $\widetilde{p}$ is $\widehat{p}=\int g(x)w^t(x)dx$. The second integral, on the contrary, depends only on $\widetilde{f}$ but not on $\widetilde{p}$. It can be rewritten as
\begin{align*}
&-\int g(x)w^t(x)\log\mathcal{N}\widetilde{f}(x)dx=-\int\int g(x)w_t(x)K_h(x-u)\log\widetilde{f}(u)dudx\\
&=-\int\left(\int K_h(u-x)g(x)w^t(x)dx\right)\log\widetilde{f}(u)du\\
&=-\frac{1}{\alpha^{t+1}}\int f^{t+1}(u)\log\widetilde{f}(u)du\\
&=\frac{1}{\alpha^{t+1}}\int f^{t+1}(u)\log\frac{f^{t+1}(u)}{\widetilde{f}(u)}du-\frac{1}{\alpha^{t+1}}
\int f^{t+1}(u)\log f^{t+1}(u)du.
\end{align*}
The first term in the above is the Kullback-Leibler divergence between $f^{t+1}$ and $\widetilde{f}$
scaled by $\alpha^{t+1}$, which is minimized at $f^{t+1}$,
i.e., for $\widehat{f}=f^{t+1}$. Since the second term does not depend on $\widetilde{f}$ at all, we arrive at the needed conclusion.
\end{proof}

The above suggests that the following algorithm can be used to estimate the parameters of the model \eqref{model2}. First, we start with initial values $p_{0}, f^{0}$ at the step $t=0$. Then, for any $t=1,2,\ldots$

\begin{itemize}
\item Define the weight 
\begin{equation}\label{weight_eq}
w^{t}(x)=\frac{p^{t}{\cal N}f^{t}(x)}{(1-p^{t})f_{0}(x)+p^{t}{\cal N}f^{t}(x)} 
\end{equation}
 \item Define the updated probability 
 \begin{equation}\label{p_eq}
p^{t+1}=\int g(x)w^t(x)dx
\end{equation}
\item Define
\begin{equation}\label{func_eq}
f^{t+1}(u)=\alpha^{t+1}\int K_h(u-x)g(x)w^t(x)dx
\end{equation}
\end{itemize} 
\begin{Remark}
Note that the proposed algorithm is an MM (majorization-minimization) and not a true EM algorithm.  MM algorithms are commonly used whenever optimization of a difficult objective function is best avoided and a series of simpler objective functions is optimized instead. A general introduction to MM algorithms is available in, for example,  \cite{hunter2004tutorial}. As a first step, let $(p^{t},f^{t})$ denote the current parameter values in our iterative algorithm. The main goal is to obtain a new functional $b^{t}(p,f)$ such that, when shifted by a constant, it majorizes $\ell(p,f)$. In other words, there must exist a constant $C^{t}$ such that, for any $(p,f)$ $b^{t}(p,f)+C^{t}\ge \ell(p,f)$ with equality when $(p,f)=(p^{t},f^{t})$. The use of $t$ as a superscript in this context indicates that the definition of the new functional $b^{t}(p,f)$ depends on the parameter values $(p^{t},f^{t})$; these change from one iteration to the other. 

In our case, we define a functional 
\begin{align}\label{maj_func}
& b^{t}(\tilde p,\tilde f)=-\int g(x)[(1-\omega^{t}(x))\log (1-\tilde p)+\omega^{t}(x)\log \tilde p]\,dx\\
&-\int g(x)\omega^{t}(x)\log {\cal N}\tilde f(x)\,dx\nonumber;
\end{align}
note that the dependence on $f^{t}$ is through weights $\omega^{t}$. 
From the proof of the Theorem \eqref{descent:property}, it follows that, for any argument $(\tilde p,\tilde f)$ we have 
\[
\ell(\tilde p,\tilde f) -\ell(p^{t},f^{t})\le b^{t}(\tilde p,\tilde f)-b^{t}(p^{t},f^{t}).
\]
This means, that $b^{t}(\tilde p,\tilde f)$ is a majorizing functional; indeed, it is enough to select the constant $C^{t}$ such that $C^{t}= \ell(p^{t},f^{t})-b^{t}(p^{t},f^{t})$. In the proof of the Theorem \eqref{descent:property} it is the series of functionals $b^{t}(\tilde p,\tilde f)$ (note that they are different at each step of iteration) that is being minimized with respect to $(\tilde p,\tilde f)$, and not the original functional $\ell(\tilde p,\tilde f)$. This, indeed, establishes that our algorithm is an MM algorithm. 
\end{Remark}
The following Lemma shows that the sequence $\xi_{t}=\ell(p^t,f^t)$, defined by our algorithm, also has a non-negative limit (which is not necessarily a global minimum of $\ell(p,f)$). 
\begin{Lemma} \label{lemma:functional_positivity}
There exists a finite limit of the  sequence $\xi_{t}=\ell(p^t,f^t)$ as $t\rightarrow \infty$:
\[
L:=\lim_{t\rightarrow\infty}\xi_{t}
\]
for some $L\ge 0$.
\end{Lemma}
\begin{proof}[Proof of Lemma \eqref{lemma:functional_positivity}]
First, note that $\xi_{t}$  is a non-increasing sequence for any integer $t$ due to the Theorem \eqref{descent:property}. Thus, if we can show that it is bounded from below by zero, the proof will be finished.  Then, the functional $\ell(p^t,f^t)$ can be represented as 
\begin{align*}
&\ell(p^t,f^t)=KL(g(x), (1-p^t)f_{0}(x)+p^t{\cal N}f^t(x))+\int g(x)\,dx\\
&-\int [(1-p^t)f_{0}(x)+p^t {\cal N}f^t(x)]\,dx\\
&=KL(g(x), (1-p^t)f_{0}(x)+p^t{\cal N}f^t(x))+1-(1-p^t)-p^t\int {\cal N}f^t(x)\,dx\\
&= KL(g(x),(1-p^t)f_{0}(x)+p^t{\cal N}f^t(x))+p^t\left[1-\int {\cal N}f^t(x)\,dx\right]
\end{align*}
Now, since $K$ is a proper density function, by Jensen's inequality, 
\begin{align*}
&{\cal N}f^t(x)=\exp{\left\{\int K_{h}(x-u)\log f^t(u)\,du\right\}}\\
&\le \int K_{h}(x-u)f_t(u)\,du\equiv {\mathcal S}f^t(x).
\end{align*} Moreover, using Fubini's theorem, one can easily show that $\int {\mathcal S}f^t(x)\,dx=1$ since $f^t$ is a proper density function. Therefore, one concludes easily that $\int {\cal N}f^t(x)\,dx\le \int {\mathcal S}f^t(x)\,dx=1$. Thus,  $\ell(p^t,f^t)\ge 0$ is non-negative due to non-negativity of the Kullback-Leibler distance.  
\end{proof} It is, of course, not clear directly from the \eqref{lemma:functional_positivity} if the sequence $(p^{t},f^{t})$, generated by this algorithm, also converges. Being able to answer this question requires establishing a lower semicontinuity property of the functional $\ell(p,f)$. Some additional requirements have to be imposed on the kernel function $K$ in order to obtain the needed result that is given below. We denote $\Delta$ the domain of the kernel function $K$.   
\begin{theorem}\label{theorem:lower_semicontinuity}
Let the kernel $K:\Delta\rightarrow \bR$ be bounded from below and Lipschitz continuous with the Lipschitz constant $C_{K}$.  Then, the minimizing sequence $(p^{t},f^{t})$ converges to $(p^{*}_{h},f^{*}_{h})$ that depends on the bandwidth $h$ such that $L=l(p^{*}_{h},f^{*}_{h})$.   
\end{theorem}
\begin{proof}
We prove this result in two parts. First, let us introduce a subset of functions $B=\{{\mathcal S}\phi:0\le \phi\in L_{1}(\Delta), \int \phi=1\}$. Such a subset represents all densities on a closed compact interval that can be represented as linearly smoothed integrable functions. Every function $f_{t}$ generated in our algorithm except, perhaps, the initial one, can clearly be represented in this form. This is because, at every step of iteration, $f^{t+1}(x)=\alpha^{t+1}\int K_{h}(x-u)g(u)w^{t}(u)\,du=\int K_{h}(x-u)\phi(u)\,du$ where $\phi(u)=\alpha^{t+1}g(u)w^{t}(u)$.  Moreover, we observe that $\int \phi(u)\,du=\alpha^{t+1}\int g(u)w^{t}(u)\,du=\alpha^{t+1}p^{t+1}$.  Next, one concludes, by using Fubini theorem  that, for any $t=1,2,\ldots$ 
 \[
 \int f^{t+1}(x)\,dx=\alpha^{t+1}\int g(u)w^{t}(u)\left[\int K_{h}(x-u)\,dx\right]\,du=1.
 \] Since the iteration step $t$ in the above is arbitrary, we established that $\alpha^{t}p^{t}=1$ and, therefore, $\int \phi(u)\,du=1$. Next, since the kernel function $K$ is bounded from below, we can easily claim that for every $f\in B$ $f=\int K_{h}(x-u)\phi(u)\,du\ge \inf_{x\in \Delta} K_{h}(x-u) \int \phi(u)\,du=\inf_{x\in \Delta} K_{h}(x-u)>0$ and, therefore, every function in the set $B$ is bounded from below. If the kernel function is Lipschitz continuous on $\Delta$ it is clearly bounded from above by some positive constant $M:\sup_{x\in \Delta} K(x)<M$. Thus, every function $f\in B$ satisfies $f(x)\le M<\infty$. This implies that the set $B$ is uniformly bounded. Also, by definition of set $B$, for any two points $x,y\in \Delta$ we have 
\[
\vert f(x)-f(y)\vert \le \int |K_{h}(x-u)-K_{h}(y-u)\vert \phi(u)\,du\le C_{K}\vert x-y\vert 
\]
where the constant $C_{K}$ depends on the choice of kernel $K$ but not on the function $f$. This establishes the equicontinuity of the set $B$. Therefore, by Arzela-Ascoli theorem the set of functions $B$ is a compact subset of $C(\Delta)$ with a $\sup$ metric.

Since for every $t=2,3,\ldots$ $f^{t}\in B$, by Arzela-Ascoli theorem we have a subsequence $f^{t_{k}}\rightarrow f^{*}_{h}$ as $k\rightarrow \infty$ uniformly over $\Omega$. Since for every $t=1,2,\ldots$ $p^{t}$ is bounded between $0$ and $1$, there exists, by Bolzano-Weierstrass theorem, a subsequence $p^{t_{k}}\rightarrow p^{*}_{h}$ as $k\rightarrow \infty$ in the usual Euclidean metric. Consider a Cartesian product space $\{(p,f)\}$ where every $p\in [0,1]$ and $f\in C(\Delta)$. To define a metric on such a space we introduce an $m$-product of individual metrics for some non-negative $m$. This means that, if the first component space has a metric $d_{1}$ and the second $d_{2}$, the metric on the Cartesian product is $(\vert d_{1}\vert ^{m}+\vert d_{2}\vert ^{m})^{1/m}$ for some non-negative $m$. For example, the specific case $m=0$ corresponds to $\vert d_{1}\vert +\vert d_{2}\vert$ and $m=\infty$ corresponds to $\max (d_{1},d_{2})$. For such an $m$-product metric, clearly, we have a subsequence $(p^{t_{k}},f^{t_{k}})\rightarrow (p^{*}_{h},f^{*}_{h})$ that converges to $(p^{*}_{h},f^{*}_{h})$ in the $m$-product metric. Without loss of generality, assume that the subsequence coincides with the whole sequence $(p^{t},f^{t})$. Of course, such a sequence $(p^{t},f^{t})\in [0,1]\times C(\Delta)$ for any $t$. 

Now, that we know that there is always a converging sequence $(p^{t},f^{t})$, we can proceed further.
Since each $f^{t}$ is bounded away from zero and from above, then so is the limit function $f^{*}_{h}(x)$ in the limit $(p^{*}_{h},f^{*}_{h})$. This implies that $(p^{t},\log f^{t})\rightarrow (p^{*}_{h},\log f^{*}_{h})$ uniformly in the $m$-product topology as well and the same is true also for $(p^{t},{\mathcal S}\log f^{t})$. Analogously, the uniform convergence follows also in $(p^{t},{\cal N}f^{t})\rightarrow (p^{*}_{h},{\cal N}f^{*}_{h})$; moreover, $(1-p^{t})f_{0}+p^{t}{\cal N}f^{t}\rightarrow (1-p^{*}_{h})f_{0}+p^{*}_{h}{\cal N}f^{*}_{h}$ uniformly in the $m$-product topology. Since the function $\psi(t)=-\log t+t-1\ge 0$, Fatou Lemma implies that 
\begin{align*}
&\int g(x)\psi((1-p^{*}_{h})f_{0}(x)+p^{*}_{h}{\cal N}f^{*}_{h}(x))\,dx\\
&\le \liminf\int g(x)\psi((1-p^{t})f_{0}(x)+p^{t}{\cal N}f^{t}(x))\,dx.
\end{align*}
The lower semicontinuity of the functional $\ell(p,f)$ follows immediately and with it the conclusion of the Theorem \eqref{theorem:lower_semicontinuity}.  
\end{proof}

\section{An empirical version of our algorithm}\label{emp_sec}

In practice, the number of observations $n$ sampled from the target density function $g$ is finite. This necessitates the development of the empirical version of our algorithm that can be implemented in practice. Many proof details here are similar to proofs of properties of the algorithm we introduced in the previous chapter. Therefore, we will be brief in our explanations. Denote the empirical cdf of the observations $X_{i}$, $i=1,\ldots,n$ $G_n(x)$ where $G_{n}(x)=\frac{1}{n}\sum_{i=1}^nI_{X_i\le x}
$. Then, we define a functional 
\begin{align*}\label{emp.func}
&l_n(f,p)=-\int \log ((1-p)f_{0}(x)+p{\cal N}f(x))\,dG_{n}(x)\\
&\equiv -\sum_{i=1}^{n}\log((1-p)f_0(X_{i})+p\mathcal{N}f(X_{i})).
\end{align*}
The following analogue of the Lemma \eqref{lemma:divergence} can be easily established. 
\begin{Lemma}\label{empirical:lemma:divergence}
For any pdf $\widetilde{f}$ and $\widetilde{p}\in(0,1)$,
\begin{align*}
&l_n(\widetilde{f},\widetilde{p})-l_n(f,p)\\
&\le -\int\left[(1-w(x))\log\left(\frac{1-\widetilde{p}}{1-p}\right)+
w(x)\log\left(\frac{\widetilde{p}\mathcal{N}\widetilde{f}(x)}{p\mathcal{N}f(x)}\right)\right]dG_n(x),
\end{align*}
where the weight $w(x)=\frac{p\mathcal{N}f(x)}{(1-p)f_0(x)+p\mathcal{N}f(x)}$.
\end{Lemma}
The proof is omitted since it is almost exactly the same as the proof of the Lemma \eqref{lemma:divergence}.

Now we can define the empirical version of our algorithm. Denote $(p_{n}^t,f_{n}^t)$ values of the density $f$ and probability $p$ at the iteration step $t$. Define the weights as $w_{n}^t(x)=\frac{p_{n}^t\mathcal{N}f_{n}^t(x)}{(1-p_{n}^t)f_0(x)+p_{n}^t\mathcal{N}f_{n}^t(x)}$. We use the subscript $n$ everywhere intentionally to stress that these quantities depend on the sample size $n$. For the next step, define $(p_{n}^{t+1},f_{n}^{t+1})$ as
\begin{align}
&p_{n}^{t+1}=\int w_{n}^{t}(x)dG_n(x)=\frac{1}{n}\sum_{i=1}^n w_{n}^{t}(X_i) \nonumber\\ 
&f_{n}^{t+1}(x)=\alpha_{n}^{t+1}\int K_h(x-u)w_{n}^t(u)dG_n(u) \nonumber\\ 
&=\frac{\alpha_{n}^{t+1}}{n}\sum_{i=1}^n K_h(x-X_i)w_{n}^{t}(X_i)\nonumber
\end{align}
where $\alpha_{n}^{t+1}$ is a normalizing constant such that $f_{n}^{t+1}$ is a valid pdf.
Since $\int K_h(X_i-u)du=1$ for $i=1,\ldots,n$, we get
\[
1=\int f_{n}^{t+1}(u)du=\frac{\alpha_{n}^{t+1}}{n}\sum_{i=1}^nw_{n}^{t}(X_i),
\]
and hence,
\[
\alpha_{n}^{t+1}=\frac{n}{\sum_{i=1}^n w_{n}^{t}(X_i)}.
\]
The following result establishes the descent property of the empirical version of our algorithm. 
\begin{theorem}\label{empirical:descent:property}
For any $t\ge0$, $\ell_n(p_{n}^{t+1},f_{n}^{t+1})\le\ell_n(p_{n}^{t},f_{n}^{t})$.
\end{theorem}

The proof of this result follows very closely the proof of the Theorem \eqref{descent:property} and is also omitted for brevity. 

\begin{Remark}
As before, the empirical version of the proposed algorithm is an MM (majorization - minimization) and not a true EM algorithm.
More specifically, we can show that there exists another functional $b^{t}_n(p,f)$ such that, when shifted by a constant, it majorizes $l_{n}(p,f)$. It is easy to check that such a functional is 
\begin{align}\label{maj_func_emp}
& b^{t}_n(\tilde p,\tilde f)=-\int [(1-\omega^{t}_{n}(x))\log (1-\tilde p)+\omega^{t}_{n}(x)\log \tilde p]\,dG_{n}(x)\\
&-\int \omega^{t}_{n}(x)\log {\cal N}\tilde f(x)\,dG_{n}(x)\nonumber. 
\end{align}
Note that in the proof of the Theorem \eqref{empirical:descent:property} it is the series of functionals $b^{t}_{n}(\tilde p,\tilde f)$ that is being minimized with respect to $(\tilde p,\tilde f)$, and not the original functional $l_{n}(\tilde p,\tilde f)$.

\end{Remark}

As before, we can also show that the sequence $\ell_{n}(p_{n}^{t},f_{n}^{t})$ generated by our algorithm does not only possess the descent property but is also bounded from below. 
\begin{Lemma} \label{lemma:empirical_limit}
There exists a finite limit of the  sequence $\xi_{n}^{t}=\ell_n(p_{n}^{t},f_{n}^{t})$ as $t\rightarrow \infty$:
\[
L_{n}=\lim_{t\rightarrow\infty}\xi_{n}^{t}
\]
for some $L_{n}\ge 0$.
\end{Lemma}
The proof is almost exactly the same as the proof of the Lemma \eqref{lemma:functional_positivity} and is omitted in the interest of brevity. 
Finally, one can also show that the sequence $(p^{t}_{n},f_{n}^{t})$ generated by our algorithm converges to $(p^{*}_{n}, f^{*}_{n})$ such that $L_{n}=l_{n}(p^{*}_{n},f^{*}_{n})$. The proof is almost the same as that of the Theorem \eqref{theorem:lower_semicontinuity} and is omitted for conciseness. 

\section{Simulations and comparison}\label{sim}

\begin{figure}[!htb]
  \centering
  \subfigure[Fitted mixture density]{
  \includegraphics[width=0.47\columnwidth]{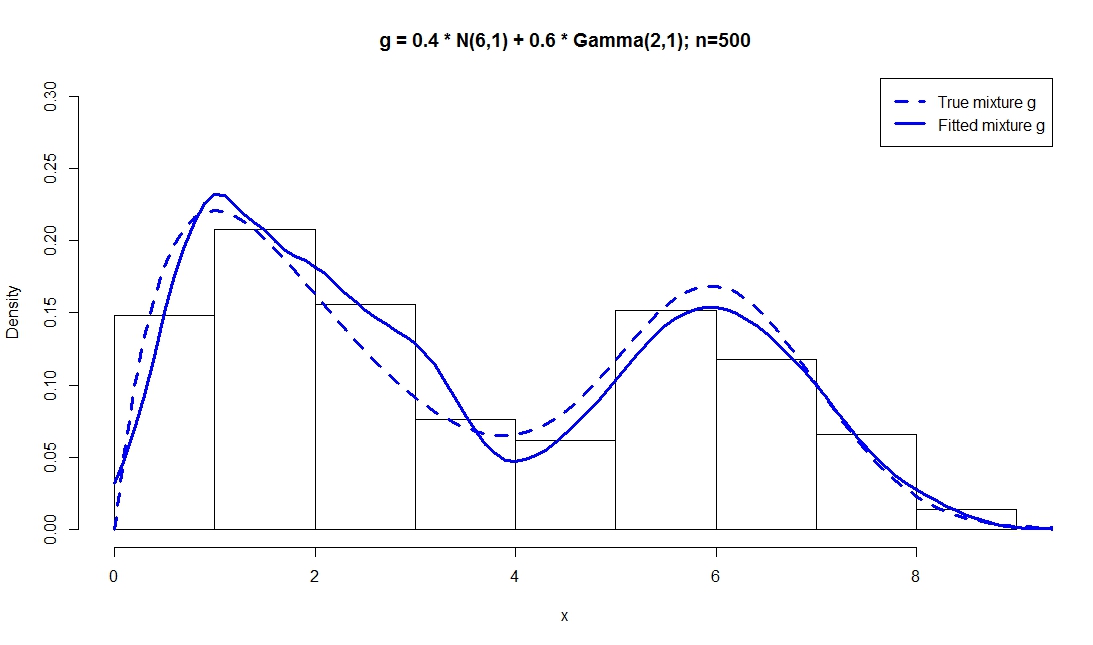}
  \label{fig:Gamma:mixture}}
  \subfigure[Fitted unknown component density]{
  \includegraphics[width=0.47\columnwidth]{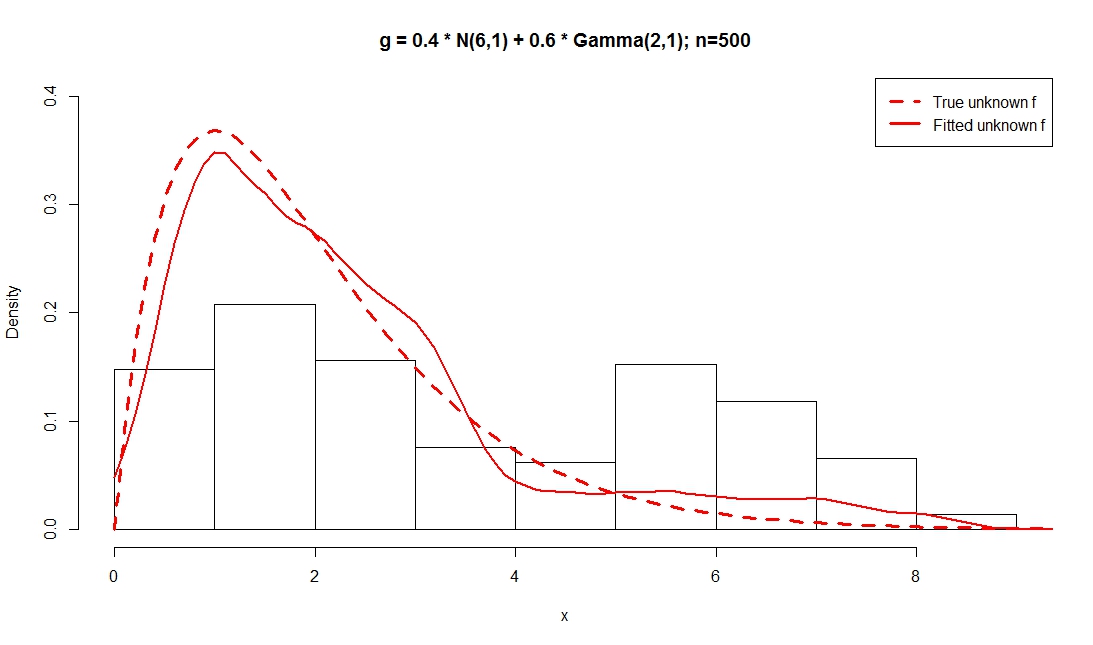}
  \label{fig:Gamma:unknown}}
  \caption{Mixture of Gaussian (6,1) and Gamma (2,1)}
\end{figure}

In this section, we will use the notation $I_{[x>0]}$ for the indicator function of the positive half of the real line and $\phi(x)$ for the standard Gaussian distribution. For our first experiment, we generate $n$ independent and identically distributed observations from a two component normal gamma mixture with the density $g(x)$ defined as $g(x)=(1-p)f_{0}(x)+pf(x).$ Thus, the known component is $f_{0}(x)=\frac{2}{\sigma}\phi\left(\frac{x-\mu}{\sigma}\right)I_{[x>0]}$  while the unknown component is $Gamma(\alpha,\beta)$ , i.e., $f(x)=\frac{\beta^{\alpha}}{\Gamma(\alpha)}x^{\alpha-1}e^{-\beta x}I_{[x>0]}$. Note that we truncate the normal distribution so that it stays on the positive half of the real line. We choose the sample size $n=500$, the probability $p=0.6$, $\mu=6$, $\sigma=1$, $\alpha=2$ and $\beta=1$. The initial weight is $p_0=0.2$ and the initial assumption for the unknown component distribution is $Gamma(4,2)$. The rescaled triangular function $K_h(x)=\frac{1}{h}\left(1-\frac{|x|}{h}\right)I(|x|\le h)$ is used as the kernel function. We use a fixed bandwidth throughout the sequence of iterations and this fixed bandwidth is selected according to the classical Silverman's rule of thumb that we describe here briefly for completeness;  for more details, see \cite{silverman1986density}. Let SD and IQR be the standard deviation and interquartile range of the data, respectively. Then, the bandwidth is determined as $h=0.9\min\left\{SD, \frac{IQR}{1.34}\right\}n^{-1/5}.$ We use the absolute difference $|p_{n}^{t+1}-p_{n}^{t}|$ as a stopping criterion; at every iteration step, we check if this difference is below a small threshold value $d$ that depends on required precision. If it is, the algorithm is stopped. The analogous rule has been described for classical parametric mixtures in \cite{mclachlan2004finite}. In our setting, we use the value $d=10^{-5}$. The computation ends after $259$ iterations, with an estimate $\hat{p}=0.6661$; the Figure \ref{fig:Gamma:mixture} shows the true and estimated mixture density function $g(x)$ while the Figure \ref{fig:Gamma:unknown} shows both true and estimated second component density $f$.  Both figures show a histogram of the observed target distribution $g(x)$ in the background. Both the fitted mixture density $\hat g(x)$ and the fitted unknown component density function $\hat f(x)$ are quite close to their corresponding true density functions everywhere.


We also analyze performance of our algorithm in terms of the mean squared error (MSE) of estimated weight $\hat p$ and the mean integrated squared error (MISE) of $\hat f$. We will use two models for this purpose. The first model is the normal exponential model where the (known) normal component is the same as before while the second (unknown) component is an exponential density function $f(x)=\lambda e^{-\lambda x}I_{[x>0]}$ with $\lambda=0.5$; the value of $p$ used is $p=0.6$. The second model is the same normal-gamma model as before. For each of the two models, we plot MSE of $\hat p$ and MISE of $\hat f$ against the true $p$ for sample sizes $n=500$ and $n=1000$. Here, we use $30$ replications.  The algorithm appears to show rather good performance even for the sample size $n=500$. Note that MISE of the unknown component $f$ seems to decrease with the increase in $p$. Possible reason for this is the fact that, the larger $p$ is, the more likely it is that we are sampling from the unknown component and so the number of observations that are actually generated by $f$ grows; this seems to explain better precision in estimation of $f$ when $p$ is large.

\begin{figure}[!htb]
  \centering
  \subfigure[Normal-Exponential mixture]{
  \includegraphics[width=0.48\columnwidth]{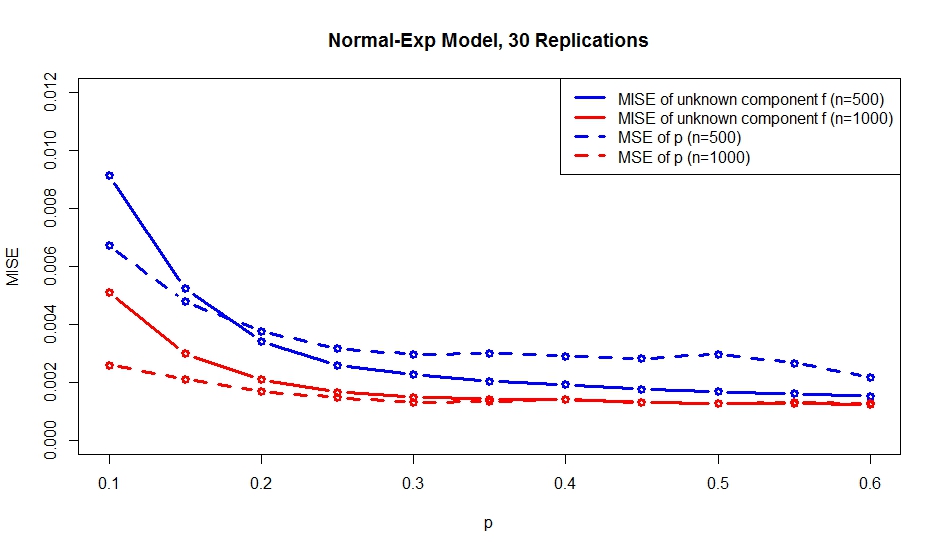}
  \label{fig:MSE:Exp}}
  \subfigure[Normal-Gamma mixture]{
  \includegraphics[width=0.48\columnwidth]{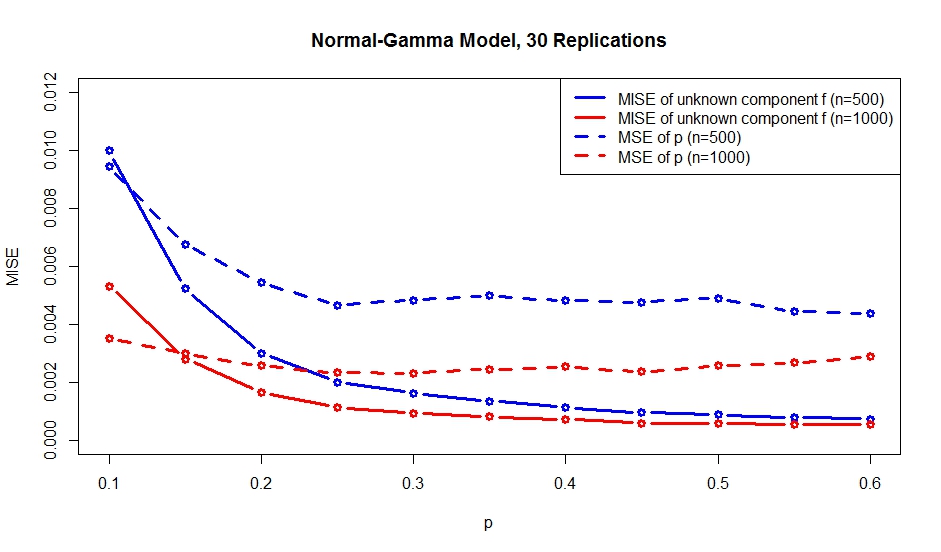}
  \label{fig:MSE:Gamma}}
  \caption{MISE of $\hat f$ and MSE of $\hat p$}
\end{figure}

Another important issue in practice is, of course, the bandwidth selection. Earlier, we simply used a fixed bandwidth selected using the classical Silverman's rule of thumb. In general, when the unknown density is not likely to be normal, the use of Silverman's rule may be a somewhat rough approach. Moreover, in an iterative algorithm, every successive step of iteration brings a refined estimate of the unknown density component; therefore, it seems a good idea to put this knowledge to use. Such an idea was suggested earlier in \cite{chauveau2015semi}. Here we suggest using a version of the $K$-fold cross validation method specifically adopted for use in an iterative algorithm. First, let us suppose we have a sample $X_{1},\ldots,X_{n}$ of size $n$; we begin with randomly partitioning it into $K$ approximately equal subsamples. For ease of notation, we denote each of these subsamples $X^{k}$, $k=1,\ldots,K$. Randomly selecting one of the $K$ subsamples, it is possible to treat the remaining $K-1$ subsamples as a training dataset and the selected subsample as the validation dataset. We also need to select a grid of possible bandwidths. To do so, we compute the preliminary bandwidth $h_{s}$ first according to the Silverman's rule of thumb; the bandwidth grid is defined as lying in an interval $[h_s -l,h_s +l]$ where $2*l$ is the range of bandwidths we plan to consider. Within this interval, each element of the bandwidth grid is computed as $h_{i}=h_s\pm \frac{i}{M}l$, $i=0,1,\ldots,M$ for some positive integer $M$. At this point, we have to decide whether a fully iterative bandwidth selection procedure is necessary. It is worth noting that a fully iterative bandwidth selection algorithm leads to the situation where the bandwidth changes at each step of iteration. This, in turn, implies that the monotonicity property of our algorithm derived in Theorem \eqref{empirical:descent:property} is no longer true. To reconcile these two paradigms, we implement the following scheme. As in earlier simulations, we use the triangular smoothing kernel. First, we iterate a certain number of times $T$ to obtain a reasonably stable estimate of the unknown $f$; if we do it using the full range of the data, we denote the resulting estimate $\hat f_{nh}^{T}(x)=\frac{\alpha_{n}^{T}}{n}\sum_{i=1}^n K_h(x-X_i)w_{n}^{T-1}(X_i)$. Integrating the resulting expression, we can obtain the squared $L_{2}$-norm of $\hat f_{nh}^{T}(x)$ as 
\[
||\hat f_{nh}^{T}||^{2}_{2}=\int\left[\frac{\alpha_{n}^{T}}{n}\sum_{i=1}^{n}w_{n}^{T-1}(X_i) K_h(x-X_i)\right]^2\,dx. 
\]  
For each of $K$ subsamples of the original sample, we can also define a ``leave-$k$th subsample out" estimator of the unknown component $f$ as $\hat f_{nh,-X_k}^{T}(x)$, $k=1,\ldots,K$ obtained after $T$ steps of iteration. At this point, we can define the CV optimization criterion as (see, for example, \cite{eggermont2001maximum})as 
\[
CV(h)=||\hat f_{nh}^{T}||^{2}_{2}-\frac{2}{n}\sum_{k=1}^{K}\sum_{x_{i}\in X_k}{\hat f}^{T}_{nh,-X_k}(x_{i}).
\]
Finally, we select 
\[
h^{*}=argmin\, CV(h)
\]
as a proper bandwidth. Now, we fix the bandwidth $h^{*}$ and keep it constant beginning with the iteration step $T+1$ until the convergence criterion is achieved and the process is stopped.  An example of a cross validation curve of $CV(h)$ is given in Figure \eqref{fig:BandwidthCV}. Here, we took a sample of size $500$ from a mixture model with a known component of $N(6,1)$, an unknown component of $Gamma(2,1)$ and a mixing proportion $p=0.5$; we also chose  $K=50$, $l=0.4$, $M=10$, and $T=5$. We tested the possibility of using larger number of iterations before selecting the optimal bandwidth $h$; however, already $T=10$ results in the selection of $h^{*}$ close to zero. We believe that the likeliest reason for that is the overfitting of the estimate of the unknown component $f$. The minimum of $CV(h)$ is achieved at around $h=0.68$.  Using this bandwidth and running the algorithm until the stopping criterion is satisfied, gives us the estimated mixing proportion $\hat p=0.497$. As a side remark, in this particular case the Silverman's rule of thumb gives a very similar estimated bandwidth $\hat h=0.72$. 
\begin{figure}[!htb]
  \centering
   \includegraphics[width=0.6\columnwidth]{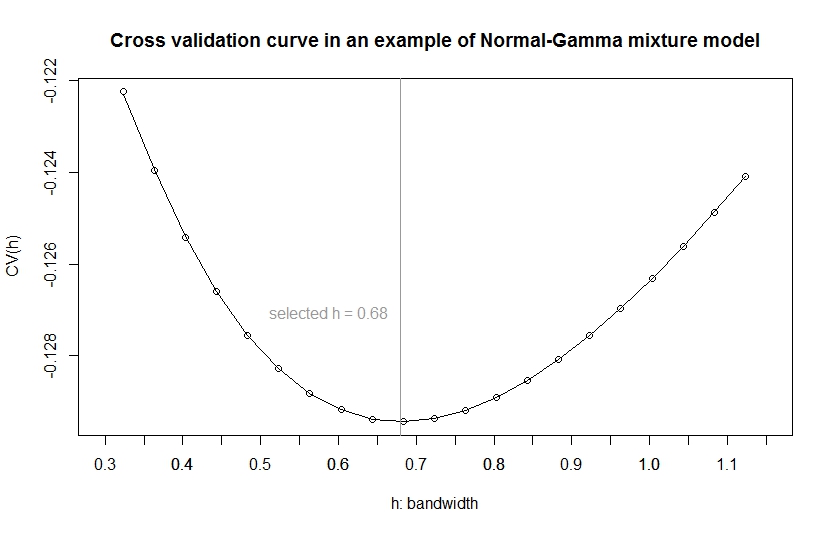}
  \caption{A plot of $CV(h)$ used for bandwidth selection}
  \label{fig:BandwidthCV}
\end{figure}

As a last step, we want to compare our method with the symmetrization method of \cite{bordes_delmas}. To do this, we will use a  normal-normal model since the method of \cite{bordes_delmas} is only applicable when an unknown component belongs to a location family. Although such a model does not satisfy the sufficient criterion of the Lemma \eqref{identifiability}, it satisfies the necessary and sufficient identifiability criterion given in Lemma $4$ of \citet{patra2015estimation} (see also Remark $3$ from the Supplement to \citet{patra2015estimation} for even clearer statement about identifiability for normal-normal models in our context); therefore, we can use it for testing purposes.  The known component has Gaussian distribution with mean $0$ and standard deviation $1$, the unknown has mean $6$ and standard deviation $1$, and we also consider two possible choices of mixture weight, $p=0.3$ and $p=0.5$. The results for two different sample sizes, $n=500$, and $n=1000$, and $200$ replications, are given below in Tables \ref{table:symmetrization} and \ref{table:ours}. Each estimate is accompanied by its standard deviation in parentheses. Note that the proper expectation here is that our method should perform similarly to the method of \cite{bordes_delmas} but not beat it, for several reasons.  First, the mean of the unknown Gaussian distribution is directly estimated as a parameter in the symmetrization method, while it is the nonparametric probability density function that is directly estimated by our method. Thus, in order to calculate the mean of the second component, we have to take an extra step when using our method and employ numerical integration. This is effectively equivalent to estimating a functional of an unknown (and so estimated beforehand) density function; therefore, somewhat lower precision of our method when estimating the mean, compared to symmetrization method, where the mean is just a Euclidean parameter, should be expected. Second, when using symmetrization method, we followed an acceptance/rejection procedure exactly as in \cite{bordes_delmas}. That procedure amounts to dropping certain ``bad" samples whereas our method keeps all the samples. Third, the method of \cite{bordes_delmas}, when estimating an unknown component, uses the fact that this component belongs to a location family - something that our method, more general in its assumptions, does not do. Keeping all of the above in mind, we can see from Tables  \eqref{table:symmetrization} and \eqref{table:ours} that both methods produce comparable results, especially when the sample size is $n=1000$. Also, as explained above, it does turn out that our method is practically as good as the method of \cite{bordes_delmas} when it comes to estimating probability $p$ and slightly worse when estimating the mean of the unknown component. However, even when estimating the mean of the unknown component, increase in sample size from $500$ to $1000$ reduces the difference in performance substantially. 

\begin{table}
\caption{Mean(SD) of estimated $p/\mu$ obtained by the symmetrization method}
\centering
\begin{tabular}{|c|c|c|}
\hline
\hline
$K=200$ & $n=500$ & $n=1000$ \\
\hline
$p=0.3/\mu = 6$ & 0.302(0.022)/5.989(0.095) &  0.302(0.016)/5.998( 0.064)  \\
\hline
$p=0.5/\mu = 6$ & 0.502(0.024)/5.999(0.067) & 0.502(0.017)/6.003(0.050)  \\
\hline
\end{tabular}
\label{table:symmetrization}
\end{table}

\begin{table}
\centering
\begin{tabular}{|c|c|c|}
\hline
\hline
$K=200$ & $n=500$ & $n=1000$ \\
\hline
$p=0.3/\mu = 6$ & 0.315(0.024)/5.772(0.238) & 0.312(0.018)/5.818(0.178)  \\
\hline
$p=0.5/\mu = 6$ & 0.516(0.026)/5.855(0.155) & 0.512(0.018)/5.883(0.117)  \\
\hline
\end{tabular}
\caption{Mean(SD) of estimated $p/\mu$ obtained by our algorithm}
\label{table:ours}
\end{table}

\section{A real data example}\label{realdata}

The acidification of lakes in parts of North America and Europe is a serious concern. In $1983$, the US Environmental Protection Agency (EPA) began the EPA National Surface Water Survey (NSWS) to study acidification as well as other characteristics of US lakes. The first stage of NSWS was the Eastern Lake Survey, focusing on particular regions in Midwestern and Eastern US. Variables measured include acid neutralizing capacity (ANC), pH, dissolved organic carbon, and concentrations of various chemicals such as iron and calcium. The sampled lakes were selected systematically from an ordered list of all lakes appearing on $1:250,000$ scale US Geological Survey topographic maps. Only surface lakes with the surface area of at least $4$ hectares were chosen. 

Out of all these variables, ANC is often the one of greatest interest. It describes the capability of the lake to neutralize acid; more specifically, low (negative) values of ANC can lead to a loss of biological resources. We use a dataset containing, among others, ANC data for a group of $155$ lakes in north-central Wisconsin. This dataset has been first  published in \cite{crawford1992modeling} in Table $1$ and analyzed in the same manuscript.  \cite{crawford1992modeling} argue that this dataset is rather heterogeneous due to the presence of lakes that are very different in their ANC within the same sample. In particular, seepage lakes, that have neither inlets nor outlets tend to be very low in ANC whereas drainage lakes that include flow paths into and out of the lake tend to be higher in ANC. Based on this heterogeneity, \cite{crawford1992modeling} suggested using an empirical mixture of two lognormal densities to fit this dataset. \cite{crawford1994application} also considered that same dataset; they suggested using a modification of Laplace method to estimate posterior component density functions in the Bayesian analysis of a finite lognormal mixture. Note that \cite{crawford1994application} viewed the number of components in the mixture model as a parameter to be estimated; their analysis suggests a mixture of either two or three components.   

The sample histogram for the ANC dataset is given on Figure 1 of \cite{crawford1994application}. The histogram is given for a log transformation of the original data $log(ANC+50)$. \cite{crawford1994application} selected this transformation to avoid numerical problems arising from maximization involving a truncation; the choice of $50$ as an additive constant is explained in more detail in \cite{crawford1994application}.  The empirical distribution is clearly bimodal; moreover, it exhibits a heavy upper tail. This is suggestive of a two-component mixture where the first component may be Gaussian while the other is defined on the positive half of the real line and has a heavy upper tail. We estimate a two-component density mixture model for this empirical distribution using two approaches. First, we follow the Bayesian approach of \cite{crawford1994application}  using the prior settings of Table $4$ in that manuscript. Switching to our framework next, we assume that the normal component is a known one while the other one is unknown. For the known normal component, we assume the mean $\mu_1 = 4.375$ and $\sigma_1 = 0.416$; these are the estimated values obtained in \cite{crawford1994application} under the assumption of two component Gaussian mixture for the original (not log transformed) data and given in their Table $4$. Next, we apply our algorithm in order to obtain an estimate of the mixture proportion and a non-parametric estimate of the unknown component to compare with respective estimates in \cite{crawford1994application}. We set the initial value of the 	mixture proportion as $p^0=0.3$ and the initial value of the unknown component as a normal distribution with mean $\mu_2^0=8$ and standard deviation  $\sigma_2^0=1$. The iterations stop when $|p^{t+1}-p^t|<10^{-4}$.  After $171$ iterations, the algorithm terminates with an estimate of mixture proportion $\hat{p}=0.4875$; for comparison purposes,  \cite{crawford1994application} produces an estimate $\hat{p}_{Bayesian}=1-0.533=0.4667$. The Figure \eqref{fig:fitmixture} shows the resulting density mixtures fitted using the method of \cite{crawford1994application} and our method against the background histogram of the log-transformed data. The Figure \eqref{fig:fitcomponent} illustrates the fitted first component of the mixture according to the method of \cite{crawford1994application} as well as the second component fitted according to both methods. Once again, the histogram of the log-transformed data is used in the background. 

\begin{figure}[!htb]
  \centering
   \includegraphics[width=0.6\columnwidth]{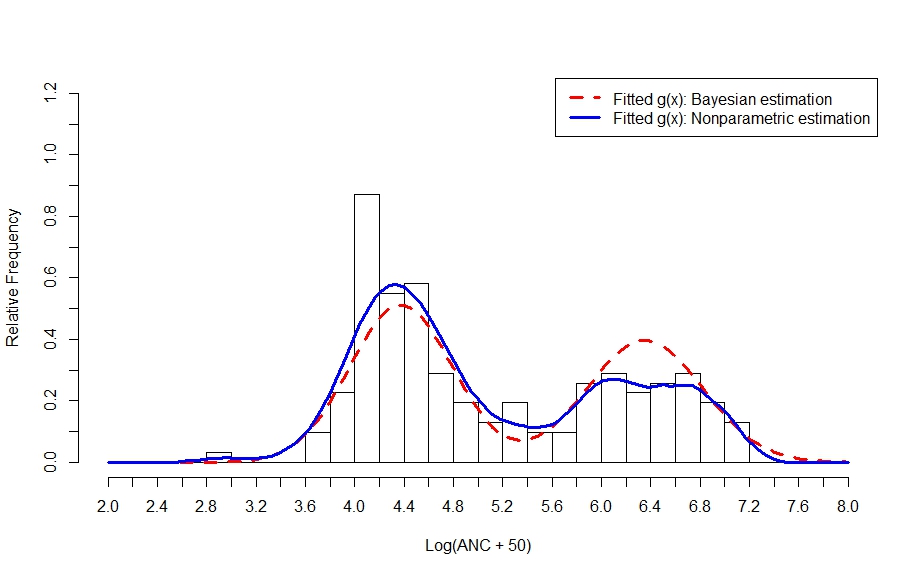}
  \caption{Fitted mixture densities}
  \label{fig:fitmixture}
\end{figure}

\begin{figure}[!htb]
  \centering
   \includegraphics[width=0.6\columnwidth]{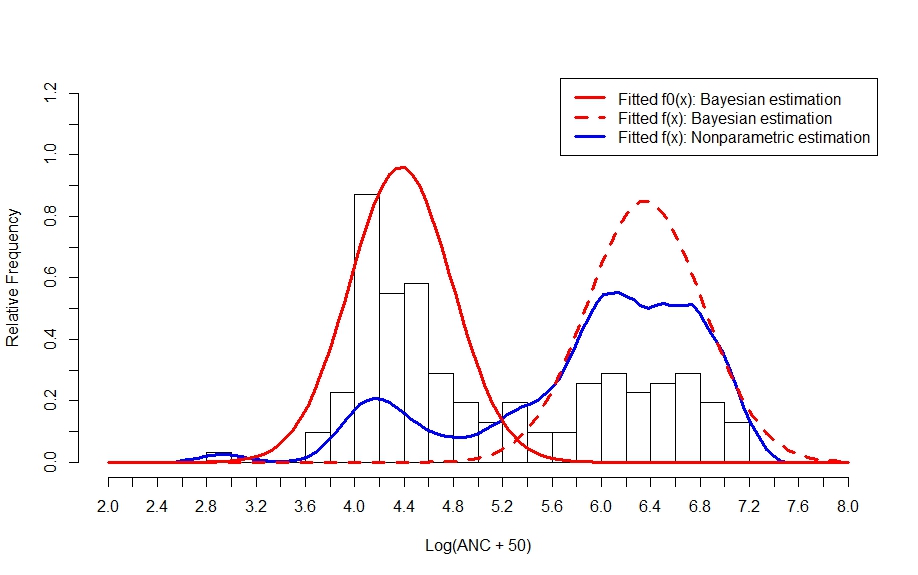}
  \caption{Fitted component densities}
  \label{fig:fitcomponent}
\end{figure}

Note that the mixture density curves based on both methods are rather similar in Figure \eqref{fig:fitmixture}. One notable difference is that the method of \cite{crawford1994application} suggests mixture with peak at the value of transformed ANC of about $6.4$ whereas our method produces a curve that seems to be following the histogram more closely in that location. The Figure \eqref{fig:fitcomponent} also seems to show that our method describes the data more faithfully than that of  \cite{crawford1994application}. Indeed, the second parametric component fitted by the method of \cite{crawford1994application} is unable to reproduce the first peak around $4.2$ at all. By doing so, the method of \cite{crawford1994application} suggests that the first peak is there only due to the first component. Our method, on the contrary, suggests that the first peak is at least partly due to the second component as well. Note that \cite{crawford1994application} discusses the possibility of a three component mixture for this dataset; results of our analysis suggest a possible presence of the third component as well based on a bimodal pattern of our fitted second component density curve. Finally, note that the method of \cite{crawford1994application} produces an estimated second component that implies a much higher second peak than the data really suggests whereas our method gives a more realistic estimate. 

\section{Discussion}\label{Discussion}

The method of estimating two component semiparametric mixtures with a known component introduced in this manuscript relies on the idea of maximizing the smoothed penalized likelihood of the available data. Another possible view of this problem is as the Tikhonov-type regularization problem (sometimes also called {\it variational regularization scheme} in optimization literature). The resulting algorithm is an MM algorithm that possesses the descent property with respect to the smoothed penalized likelihood functional. Moreover, we also show that the resulting algorithm converges under mild restrictions on the kernel function used to construct a nonlinear smoother ${\cal N}$. The algorithm also shows reasonably good numerical properties, both in simulations and when applied to a real dataset. If necessary, a number of acceleration techniques can be considered in case of large datasets;  for more details, see e.g. \cite{lange2000optimization}. 

As opposed to the symmetrization method of \cite{bordes_delmas}, our algorithm is also applicable to situations where the unknown component does not belong to any location family; thus, our method can be viewed as a more universal one of two. Comparing our method directly to that of \cite{bordes_delmas} and \citet{patra2015estimation} is a little difficult since our method is based,essentially, on perturbation of observed data the amount of which is controlled by the non-zero bandwidth $h$; thus, we arrive at what is apparently a solution different from that suggested in both \cite{bordes_delmas} and \citet{patra2015estimation}. 

There are a number of outstanding questions remaining concerning the model \eqref{model2} that will have to be investigated as a part of our future research. First, the constraint that an unknown density is defined on a compact space is, of course, convenient when proving convergence of the algorithm generated sequence; however, it would be desirable to lift it later. We believe that, at the expense of some additional technical complications, it is possible to prove all of our results when the unknown density function $f(x)$ is defined on the entire real line but has sufficiently thin tails. Second, an area that we have not touched in this manuscript is the convergence of resulting solutions. For example, a solution obtained by running an empirical version of our algorithm $(p_{n}^{*},f_{n}^{*})$ would be expected to converge to a solution $(p^{*},f^{*})$ obtained by the use of the original algorithm in the integral form. Analogously, as $h\rightarrow 0$, it is natural to expect that $(p^{*},f^{*})$ would converge to $(p,f)$ such that the identity $(1-p)f_{0}(x)+pf(x)=g(x)$ is satisfied. We expect that some recent results in optimization theory concerning Tikhonov-type regularization methods with non-metric fitting functionals (see, for example, \cite{flemming2010theory} ) will be helpful in this undertaking. Our research in this area is ongoing. 



\bibliographystyle{Chicago}%
\bibliography{reference}

\begin{thebibliography}{}

\bibitem[\protect\citeauthoryear{Bar-Lev, Enis, et~al.}{Bar-Lev
  et~al.}{1986}]{bar1986reproducibility}
Bar-Lev, S.~K., P.~Enis, et~al. (1986).
\newblock Reproducibility and natural exponential families with power variance
  functions.
\newblock {\em The Annals of Statistics\/}~{\em 14\/}(4), 1507--1522.

\bibitem[\protect\citeauthoryear{Bar-Lev and Stramer}{Bar-Lev and
  Stramer}{1987}]{Bar-Lev_Stramer1987}
Bar-Lev, S.~K. and O.~Stramer (1987).
\newblock Characterizations of natural exponential families with power variance
  functions by zero regression properties.
\newblock {\em Probability Theory and Related Fields\/}.

\bibitem[\protect\citeauthoryear{Bordes, Delmas, and Vandekerkhove}{Bordes
  et~al.}{2006}]{bordes_delmas}
Bordes, L., C.~Delmas, and P.~Vandekerkhove (2006).
\newblock Semiparametric estimation of a two-componentmixturemodel where one
  component is known.
\newblock {\em Scandinavian Journal of Statistics\/}~{\em 33}.

\bibitem[\protect\citeauthoryear{Bordes and Vandekerkhove}{Bordes and
  Vandekerkhove}{2010}]{bordes2010semiparametric}
Bordes, L. and P.~Vandekerkhove (2010).
\newblock Semiparametric two-component mixture model with a known component: an
  asymptotically normal estimator.
\newblock {\em Mathematical Methods of Statistics\/}~{\em 19\/}(1), 22--41.

\bibitem[\protect\citeauthoryear{Cai and Jin}{Cai and
  Jin}{2010}]{cai2010optimal}
Cai, T.~T. and J.~Jin (2010).
\newblock Optimal rates of convergence for estimating the null density and
  proportion of nonnull effects in large-scale multiple testing.
\newblock {\em The Annals of Statistics\/}, 100--145.

\bibitem[\protect\citeauthoryear{Chauveau, Hunter, Levine, et~al.}{Chauveau
  et~al.}{2015}]{chauveau2015semi}
Chauveau, D., D.~R. Hunter, M.~Levine, et~al. (2015).
\newblock Semi-parametric estimation for conditional independence multivariate
  finite mixture models.
\newblock {\em Statistics Surveys\/}~{\em 9}, 1--31.

\bibitem[\protect\citeauthoryear{Cohen}{Cohen}{1967}]{cohen1967estimation}
Cohen, A.~C. (1967).
\newblock Estimation in mixtures of two normal distributions.
\newblock {\em Technometrics\/}~{\em 9\/}(1), 15--28.

\bibitem[\protect\citeauthoryear{Crawford}{Crawford}{1994}]{crawford1994application}
Crawford, S.~L. (1994).
\newblock An application of the laplace method to finite mixture distributions.
\newblock {\em Journal of the American Statistical Association\/}~{\em
  89\/}(425), 259--267.

\bibitem[\protect\citeauthoryear{Crawford, DeGroot, Kadane, and Small}{Crawford
  et~al.}{1992}]{crawford1992modeling}
Crawford, S.~L., M.~H. DeGroot, J.~B. Kadane, and M.~J. Small (1992).
\newblock Modeling lake-chemistry distributions: Approximate bayesian methods
  for estimating a finite-mixture model.
\newblock {\em Technometrics\/}~{\em 34\/}(4), 441--453.

\bibitem[\protect\citeauthoryear{Day}{Day}{1969}]{day1969estimating}
Day, N.~E. (1969).
\newblock Estimating the components of a mixture of normal distributions.
\newblock {\em Biometrika\/}~{\em 56\/}(3), 463--474.

\bibitem[\protect\citeauthoryear{Efron}{Efron}{2012}]{efron2012large}
Efron, B. (2012).
\newblock {\em Large-scale inference: empirical Bayes methods for estimation,
  testing, and prediction}, Volume~1.
\newblock Cambridge University Press.

\bibitem[\protect\citeauthoryear{Eggermont, LaRiccia, and LaRiccia}{Eggermont
  et~al.}{2001}]{eggermont2001maximum}
Eggermont, P. P.~B., V.~N. LaRiccia, and V.~LaRiccia (2001).
\newblock {\em Maximum penalized likelihood estimation}, Volume~1.
\newblock Springer.

\bibitem[\protect\citeauthoryear{Flemming}{Flemming}{2010}]{flemming2010theory}
Flemming, J. (2010).
\newblock Theory and examples of variational regularization with non-metric
  fitting functionals.
\newblock {\em Journal of Inverse and Ill-Posed Problems\/}~{\em 18\/}(6),
  677--699.

\bibitem[\protect\citeauthoryear{Flemming}{Flemming}{2011}]{flemming2011generalized}
Flemming, J. (2011).
\newblock {\em Generalized Tikhonov regularization: basic theory and
  comprehensive results on convergence rates}.
\newblock Ph.\ D. thesis.

\bibitem[\protect\citeauthoryear{Hall, Neeman, Pakyari, and Elmore}{Hall
  et~al.}{2005}]{Hall_Pakyari_Elmore}
Hall, P., A.~Neeman, R.~Pakyari, and R.~Elmore (2005).
\newblock Nonparametirc inference in multivariate mixtures.
\newblock {\em Biometrika Trust\/}~{\em 3}.

\bibitem[\protect\citeauthoryear{Hall and Zhou}{Hall and
  Zhou}{2003}]{Hall_Zhou}
Hall, P. and X.~Zhou (2003).
\newblock Nonparametric estimation of component distributions in multivariate
  mixture.
\newblock {\em The Annals of Statistics\/}~{\em 31}.

\bibitem[\protect\citeauthoryear{Hofmann, Kaltenbacher, Poeschl, and
  Scherzer}{Hofmann et~al.}{2007}]{hofmann2007convergence}
Hofmann, B., B.~Kaltenbacher, C.~Poeschl, and O.~Scherzer (2007).
\newblock A convergence rates result for tikhonov regularization in banach
  spaces with non-smooth operators.
\newblock {\em Inverse Problems\/}~{\em 23\/}(3), 987.

\bibitem[\protect\citeauthoryear{Hunter and Lange}{Hunter and
  Lange}{2004}]{hunter2004tutorial}
Hunter, D.~R. and K.~Lange (2004).
\newblock A tutorial on mm algorithms.
\newblock {\em The American Statistician\/}~{\em 58\/}(1), 30--37.

\bibitem[\protect\citeauthoryear{Jin}{Jin}{2008}]{jin2008proportion}
Jin, J. (2008).
\newblock Proportion of non-zero normal means: universal oracle equivalences
  and uniformly consistent estimators.
\newblock {\em Journal of the Royal Statistical Society: Series B (Statistical
  Methodology)\/}~{\em 70\/}(3), 461--493.

\bibitem[\protect\citeauthoryear{Lange, Hunter, and Yang}{Lange
  et~al.}{2000}]{lange2000optimization}
Lange, K., D.~R. Hunter, and I.~Yang (2000).
\newblock Optimization transfer using surrogate objective functions.
\newblock {\em Journal of computational and graphical statistics\/}~{\em
  9\/}(1), 1--20.

\bibitem[\protect\citeauthoryear{Lindsay et~al.}{Lindsay
  et~al.}{1983}]{lindsay1983geometry}
Lindsay, B.~G. et~al. (1983).
\newblock The geometry of mixture likelihoods: a general theory.
\newblock {\em The annals of statistics\/}~{\em 11\/}(1), 86--94.

\bibitem[\protect\citeauthoryear{Lindsay and Basak}{Lindsay and
  Basak}{1993}]{lindsay1993multivariate}
Lindsay, B.~G. and P.~Basak (1993).
\newblock Multivariate normal mixtures: a fast consistent method of moments.
\newblock {\em Journal of the American Statistical Association\/}~{\em
  88\/}(422), 468--476.

\bibitem[\protect\citeauthoryear{McLachlan and Peel}{McLachlan and
  Peel}{2004}]{mclachlan2004finite}
McLachlan, G. and D.~Peel (2004).
\newblock {\em Finite mixture models}.
\newblock John Wiley \& Sons.

\bibitem[\protect\citeauthoryear{Patra and Sen}{Patra and
  Sen}{2015}]{patra2015estimation}
Patra, R.~K. and B.~Sen (2015).
\newblock Estimation of a two-component mixture model with applications to
  multiple testing.
\newblock {\em Journal of the Royal Statistical Society: Series B (Statistical
  Methodology)\/}.

\bibitem[\protect\citeauthoryear{Robin, Bar-Hen, Daudin, and Pierre}{Robin
  et~al.}{2007}]{robin2007semi}
Robin, S., A.~Bar-Hen, J.-J. Daudin, and L.~Pierre (2007).
\newblock A semi-parametric approach for mixture models: Application to local
  false discovery rate estimation.
\newblock {\em Computational Statistics \& Data Analysis\/}~{\em 51\/}(12),
  5483--5493.

\bibitem[\protect\citeauthoryear{Silverman}{Silverman}{1986}]{silverman1986density}
Silverman, B.~W. (1986).
\newblock {\em Density estimation for statistics and data analysis}, Volume~26.
\newblock CRC press.

\end{thebibliography}

\end{document}